\documentclass[12pt]{amsart} 

        \title{Topological modular forms and conformal nets}       
       \author{Christopher L. Douglas} 
      \address{Mathematical Institute\\ 24--29 St Giles'\\ Oxford\\ OX1 3LB, United Kingdom}
        \email{cdouglas@maths.ox.ac.uk}
      \urladdr{http://people.maths.ox.ac.uk/cdouglas}
       \author{Andr{\'e} G. Henriques}
      \address{Mathematisch Instituut\\
               Universiteit Utrecht\\
               3508 TA Utrecht, The Netherlands}
        \email{a.g.henriques@uu.nl}
      \urladdr{http://www.staff.science.uu.nl/\!\raisebox{-1mm}{~}\!henri105}  
         \date{\today}

\usepackage[top=3.9cm,right=3.2cm,left=3.2cm,bottom=3.9cm]{geometry}
\usepackage{psfrag, rotate}
\usepackage{amsfonts}
\usepackage{amssymb}
\usepackage{amsmath}
\usepackage{times}
\usepackage{comment}
\usepackage{fancyhdr}
\usepackage[all]{xy}
\usepackage{textcomp}

\usepackage{tikz}
\usetikzlibrary{matrix}

\newtheorem*{conjecture}{Conjecture}
\newtheorem*{definition}{Definition}
\newtheorem*{theorem}{Theorem}
\newtheorem*{corollary}{Corollary}
\newtheorem*{question}{Question}

\def\cA{\mathcal A}\def\cB{\mathcal B}\def\cC{\mathcal C}

\def\CC{\mathbb C}

\def\RR{\mathbb R}\def\SS{\mathbb S}

\def\ZZ{\mathbb Z}

\def\acted{\hspace{.1cm}{\setlength{\unitlength}{.30mm}\linethickness{.09mm}
                        \begin{picture}(8,8)(0,0)\qbezier(1,6)(3.5,8.3)(6,7)\qbezier(6,7)(9.5,4)(6,1)\qbezier(6,1)(3.5,-.3)(1,2)
                                                 \qbezier(1,6)(1.9,7.5)(1.2,9)\qbezier(1,6)(3,6.1)(3.8,4.4)
                        \end{picture}\hspace{.1cm}}}

\def\Aut{\mathrm{Aut}}
\def\cl{\mathit{Cliff}\!}
\def\Cl{\mathit{Cliff}\!}

\def\CN{\mathit{CN}}
\def\Diff{\mathit{Diff}}
\def\Fer{\mathit{Fer}}
\def\Hom{\mathrm{Hom}}
\def\KO{\mathit{KO}}

\def\Spin{\mathit{Spin}}
\def\String{\mathit{String}}
\def\TMF{\mathit{TMF}}
\def\MF{\mathit{MF}}
\def\cC{\mathcal{C}}
\def\vN{\mathit{vN}}
\def\Alg{\mathit{Alg}}
\def\ra{\rightarrow}
\def\Rep{\mathrm{Rep}}
\def\id{\mathrm{id}}

\begin{document}

\maketitle

\subsection*{1. Introduction}

The generalized cohomology theory $K$-theory has had an especially rich history in large part because it admits a geometric definition in terms of vector bundles.  This definition has provided a framework for a spectacular array of generalizations of $K$-theory, including equivariant $K$-theory, algebraic $K$-theory, twisted $K$-theory, $L$-theory, and $KK$-theory, among others, with applications to such diverse topics as indices of elliptic operators, classification of manifolds, topology of diffeomorphism groups, geometry of algebraic cycles, and classification of $C^*$-algebras.

Since its introduction by Hopkins and Miller, the cohomology theory $\TMF$, topological modular forms, has been seen as a higher form of $K$-theory---here `higher' has been variously interpreted to refer to categorification or taking loop spaces, though at root it refers to the increase of the exponent controlling the formal group laws associated to the cohomology theory.  As the present definitions of $\TMF$ are predominantly homotopy-theoretic in character~\cite{[HMi],[HMa],[Lu]}, it has been a persistent dream of algebraic topologists to develop a geometric definition for $\TMF$ analogous to the one for $K$-theory.  Various potential definitions have been investigated~\cite{[BDR],[HK],[Se],[ST1]}, particularly ones utilizing bundles of linear categories, or considering vector bundles on loop spaces, or involving models for quantum and conformal field theory.  A particularly concise, if as yet ill-defined, proposal for such a geometric model, is the $\Diff(S^1)$-equivariant $K$-theory of the free loop space:
\[
\TMF^*(M) \overset{?}{\simeq} K^*_{\Diff(S^1)}(LM).
\]
\noindent A $\Diff(S^1)$-action on a vector bundle, over a free loop space, can be interpreted as a shadow of an action by a category of geometric surfaces, which itself is closely related to the structures appearing in a two-dimensional conformal field theory.  Building primarily on work of Segal and Stolz-Teichner, we propose that conformal nets, a mathematical model for certain conformal field theories, could play a prominent role in a field-theoretically-inspired geometric definition of $\TMF$.

Though the $0^{\text{th}}$ group $K^0(X)$ of the generalized cohomology theory $K$-theory can be described in terms of equivalence classes of vector bundles on the space $X$, to geometrically describe all the groups $K^n(X)$ requires considering bundles of modules for the Clifford algebras $\Cl(n)$~\cite{[Kar]}.  There is a family of fundamental conformal nets, the free fermions $\Fer(n)$, that are conformal-field-theoretic analogues of the ordinary Clifford algebras.  Moreover, there is a notion of boundary condition for a conformal net that is analogous to the notion of module for an algebra.  We therefore speculate that bundles of fermionic boundary conditions will be the basic underlying objects of a geometric model for $\TMF$.  

As orientations of manifolds and vector bundles are central to the theory of ordinary homology, providing fundamental classes, Thom classes, Euler classes, and Poincar\'e duality, among other structures, so spin structures on manifolds and vector bundles are the key notions of orientation for $K$-theory.  Furthermore, string structures are the basic orientation relevant for the theory of topological modular forms; in particular, manifolds with string structures have $\TMF$-fundamental classes and $\TMF$-Euler classes.

A spin structure on $\RR^n$ can be conveniently encoded in the context of Clifford algebras as an invertible bimodule from $\Cl(n)$ to itself; analogously, a string structure on $\RR^n$ has a natural description, using conformal nets, as an invertible bimodule (called, in this context, an invertible defect) from $\Fer(n)$ to itself.  Leveraging this notion of a string structure, we can construct the basic data of the $\TMF$-Euler class of a string vector bundle.

The 8-fold periodicity of the real $K$-theory groups is intimately related to the 8-fold Morita periodicity of the real Clifford algebras, and we imagine there is a similar relationship between the 576-fold periodicity of $\TMF$ and an appropriate notion of periodicity for the free fermion nets.  Specifically, we conjecture that the net of $n$ fermions is equivalent, in the 3-category of conformal nets, to the net of $n+576$ fermions.  Using an action of the stable homotopy groups of spheres on the homotopy of the 3-category of conformal nets, we establish a lower bound on the periodicity of the fermions, proving that if $n$ fermions is equivalent to $n+k$ fermions, then 24 must divide $k$.

\subsection*{2. Topological modular forms}

We discuss a few facts about the ring of topological modular forms, particularly concerning its periodicity and relationship to the ring of ordinary modular forms.  The reader should refer to Goerss' article~\cite{[Go]} for a thorough overview of the field, and to Behrens' paper~\cite{[Be]} and Bauer's paper~\cite{[Ba]} for details about, respectively, the construction of $\TMF$ and the computation of its homotopy.

Associated to each elliptic curve $C$ there is an elliptic cohomology theory $E_C$, and in an appropriate sense, the cohomology theory $\TMF$ functions as a universal elliptic cohomology theory.  This universal property provides a map from the ring of topological modular forms, that is from the $\TMF$-cohomology of a point, to the ring of modular forms, that is to the ring of sections of certain line bundles over the moduli stack of elliptic curves.  The ring of modular forms $\MF_{\!*}$ is $\ZZ[c_4, c_6, \Delta^{\pm1}]/(c_4^3-c_6^2-1728\,\Delta)$---see for instance Deligne's computation~\cite{[De]}---and the map
\begin{equation}\nonumber
\phi: \TMF_{\!*}(\mathit{pt})\,\,\to\,\, \MF_{\!*}
\end{equation}
is a rational isomorphism.  In fact, the kernel and cokernel of the map $\phi$ are both 2- and 3-torsion groups.  Note that $\phi$ is grading-preserving, provided we take the degrees of $c_4$, $c_6$, and $\Delta$ to be 8, 12, and 24 respectively; these are twice the classical gradings of those elements.

Because the discriminant $\Delta \in \MF_{\!*}$ is invertible, the ring of modular forms is periodic of period 24; multiplication by $\Delta$ provides the periodicity isomorphism from $\MF_n$ to $\MF_{n+24}$, for any $n$.  There is no topological modular form lifting the modular form $\Delta$, and so the periodicity of modular forms is not reproduced directly in $\TMF$.  However, the power $\Delta^{24}$ does lift to a topological modular form, and the ring $\TMF_{\!*}$ is therefore periodic of period $24 \times 24 = 576$.

This situation comparing modular forms to topological modular forms, in which a classical periodicity is dilated, is entirely analogous to the situation in $K$-theory.  One construction of real $K$-theory is as follows: to each curve carrying a formal group structure isomorphic to the multiplicative formal group, there is an associated cohomology theory equivalent to complex $K$-theory; real $K$-theory is, in an appropriate sense, the universal cohomology theory mapping to all the resulting versions of complex $K$-theory.  This universal property provides a map from the coefficient ring $\KO_{\!*}(\mathit{pt})$ of real $K$-theory to the ring of sections of certain line bundles over the moduli stack of multiplicative curves.  That ring of sections is $\ZZ[a^{\pm 1}]$, where the periodicity generator $a$ has degree 4.  The map
\[
\psi: \KO_*(pt) \,\,\to\,\, \ZZ[a^{\pm1}]
\]
is a rational isomorphism, and in fact has kernel and cokernel being 2-torsion groups.  Though the generator $a$ is not in the image of the map $\psi$, the power $a^2$ does lift to a class in real $K$-theory; thus, real $K$-theory is periodic of period $4 \times 2 = 8$.

As listing all 576 coefficient groups of $\TMF$ might tax the reader's attention, we write out these groups through degree 28, to give a sense of their structure in this range and of their relation to the entire first period of the ring of classical modular forms. \bigskip

\centerline{\parbox{13.5cm}{
\centerline{\small The homotopy groups of $\TMF$ in low degrees:}\smallskip
\centerline{
\begin{tabular}{|c|c|c|c|c|c|c|c|c|c|}\hline
$\pi_0$ & $\pi_1$ & $\pi_2$ & $\pi_3$ & $\pi_4$ & \,$\pi_5$\, & $\pi_6$ & \,$\pi_7$\, & $\pi_8$ & $\pi_9$ \\\hline
\,$\ZZ[x]$\, & $\ZZ/2[x]$\hspace{.15mm} & $\ZZ/2[x]$ & $\ZZ/24$ & $\ZZ[x]$ & $0$ & $\ZZ/2$ & $0$ & $\ZZ[x]\oplus \ZZ/2$ & $\ZZ/2[x]\oplus \ZZ/2$ \\\hline
\end{tabular}}\smallskip
\centerline{
\begin{tabular}{|c|c|c|c|c|c|c|c|c|c|}\hline
$\pi_{10}$ & $\pi_{11}$ & $\pi_{12}$ & $\pi_{13}$ & $\pi_{14}$ & $\pi_{15}$ & $\pi_{16}$ & $\pi_{17}$ & $\pi_{18}$ & $\pi_{19}$ \\\hline
$\ZZ/2[x]\oplus \ZZ/3$ & $0$ & $\ZZ[x]$ & $\ZZ/3$ & \,$\ZZ/2$\, & \,$\ZZ/2$\, & \,$\ZZ[x]$\hspace{.4mm} & $\ZZ/2[x]\oplus \ZZ/2$ & $\ZZ/2[x]$ & $0$ \\\hline
\end{tabular}}\smallskip
\centerline{
\begin{tabular}{|c|c|c|c|c|c|c|c|c|}\hline
$\pi_{20}$ & $\pi_{21}$ & $\pi_{22}$ & \!$\pi_{23}$\! & $\pi_{24}$ & $\pi_{25}$ & $\pi_{26}$ & $\pi_{27}$ & $\pi_{28}$ \\\hline
$\ZZ[x]\oplus \ZZ/24$ & $\ZZ/2$ & $\ZZ/2$ & $0$ & $24\ZZ+x\ZZ[x]$ & $\ZZ/2[x]$ & $\ZZ/2[x]$ & $\ZZ/12$ & $\ZZ[x]\oplus \ZZ/2$ \\\hline
\end{tabular}}\smallskip
}}
\bigskip

\noindent This table indicates that $\pi_0(\TMF)$ is the ring $\ZZ[x]$, and the remaining homotopy groups are described as $\pi_0(\TMF)$-modules.  The element $x \in \pi_0(\TMF)$ maps, under $\phi$, to the modular form $c_4^3 \, \Delta^{-1}$.  The $24^\text{th}$ homotopy group is isomorphic to the ideal of $\ZZ[x]$ generated by $24$ and $x$; as a $\ZZ[x]$-module, it has two generators, which we might suggestively call $[24\Delta]$ and $[c_4^3]$ subject to the relation $x[24\Delta] = 24[c_4^3]$.  Besides $[24\Delta], [c_4^3] \in \pi_{24}(\TMF)$, the only elements in this range that map nontrivially into $\MF_{\!*}$ are the $\ZZ[x]$-module generators of the $\ZZ[x]$ factor in degrees $0, 4, 8, 12, 16, 20, 28$, which hit respectively $1, 2 c_6 c_4^2 \Delta^{-1}, c_4, 2c_6, c_4^2, 2c_6 c_4, 2 c_6 c_4^2 \in \MF_{\!*}$.

\subsection*{3. Conformal nets}

The geometric model for $K$-theory depends not only on the existence of the Clifford algebras $\Cl(n)$, but also on having the Clifford algebras situated in a 2-category of algebras, bimodules, and intertwiners---this 2-category provides the crucial notions of Clifford modules and Morita equivalence.  As part of their program to develop a geometric model for $\TMF$ as a space of two-dimensional quantum field theories, Stolz and Teichner~\cite{[ST1]} realized they would need a 3-category encoding some kind of `higher Clifford algebras'.  In order that this 3-category interact well with existing notions of one-dimensional quantum field theories, they hoped the 3-category would have a proscribed relationship to the 2-category of von Neumann algebras over the complex numbers.  Specifically, they asked the following:
\begin{question} (Stolz-Teichner, 2004)
Is there an interesting symmetric monoidal 3-category $\cC$ delooping the 2-category $\vN$ of von Neumann algebras, in the sense that $\Hom_{\cC}(1,1) = \vN$?
\end{question}
\noindent As it is certainly possible to take a connected deloop $B\vN$ of $\vN$, in this question `interesting' amounted to wanting there to be natural objects of $\cC$ besides the unit object, ideally a family of objects with periodicity properties analogous to those of Clifford algebras.

In the forthcoming papers~\cite{[BDH123]}, we introduce a composition operation on defects between conformal nets and show that conformal nets themselves form a 3-category delooping von Neumann algebras.  Conformal nets are one of the existing mathematical models for conformal field theory; traditionally, a conformal net is defined (with inspiration from algebraic quantum field theory) as a collection of algebras of operators associated to the subintervals of the standard circle.  There are many other formalisms for conformal field theory, for instance vertex algebras~\cite{[FBZ]}, chiral algebras~\cite{[BD]}, and algebras over (partial) operads~\cite{[Se],[Hu],[FRS]}.  Among these frameworks, conformal nets provide a particularly appropriate context for investigating the higher categorical structure of conformal field theory.  As such, we borrow terminology from conformal field theory in naming the 1-, 2-, and 3-morphisms of our 3-category $\CN$ of conformal nets; the following table lists a few of the key pieces of categorical structure, and the associated names: \bigskip

\centerline{
\begin{tabular}{|c||c|}
\hline
\multicolumn{2}{|c|}{$\phantom{\Big|}$The 3-category of conformal nets}\\\hline\hline
$\phantom{\Big|}$Objects&Conformal nets\\\hline
$\phantom{\Big|}$Arrows, $\cA\to \cB$&Defects between the nets $\cA$ and $\cB$\\\hline
$\phantom{\Big|}$Arrows from $\cA$ to the unit object 1&Boundary conditions for $\cA$\\\hline
$\phantom{\Big|}$2-morphisms between arrows,
$\cA
{\hspace{.1cm}{\setlength{\unitlength}{.50mm}\linethickness{.09mm}
\begin{picture}(8,8)(0,0)\qbezier(0,4)(4,7)(8,4)\qbezier(0,1)(4,-2)(8,1)\qbezier(3.5,4)(3.5,3)(3.5,1.5)
\qbezier(4.5,4)(4.5,3)(4.5,1.5)\qbezier(4,0.8)(4.5,1.7)(5.5,2)\qbezier(4,0.8)(3.5,1.7)(2.5,2)
\qbezier(8,1)(7.4,.2)(7.7,-.7)\qbezier(8,1)(7,1)(6.5,1.5)\qbezier(8,4)(7.4,4.8)(7.7,5.7)
\qbezier(8,4)(7,4)(6.5,3.5)\end{picture}\hspace{.1cm}}}
\cB$
&Sectors between $\cA$-$\cB$-defects\\\hline
$\phantom{\Big|}$2-morphisms from $\mathrm{id}_\cA\!:\!\cA\to \cA$ to itself& Representations of $\cA$\\\hline
$\phantom{\Big|}$3-morphisms between 2-morphisms& Intertwiners of sectors\\\hline
\end{tabular}} \bigskip

Before describing nets and defects in more detail, we mention the kind of 3-categorical structure we consider.  On the one hand, the higher categorical structure of conformal nets is somewhat stricter than a maximally weak symmetric monoidal 3-category; on the other hand, conformal nets naturally have a bit more structure than would be recorded in an ordinary 3-category.  We address both these features by building the 3-categorical structure of conformal nets as an internal bicategory in the 2-category of symmetric monoidal categories.  Such internal bicategories, described in [BDH1], distinguish two kinds of 1-morphisms and two kinds of 2-morphisms.  This is analogous to the notion of bicategory given by internal categories in the 2-category of categories, investigated by Shulman~\cite{[Sh]}, which records two distinct kinds of 1-morphisms; for instance, in the bicategory of algebras, the ring homomorphisms and the bimodules between algebras are two structurally distinct types of morphisms.  In the internal bicategory of conformal nets, the 1-morphisms are split into the homomorphisms of nets and the defects between nets, and the 2-morphisms are split into the homomorphisms of defects and the sectors between defects.

As mentioned above, the classical notion of conformal net~\cite{[Kaw],[GF],[Lo]} focuses on a collection of algebras associated to subintervals of the standard circle.  In order to facilitate the construction of the higher categorical structure of nets, we introduce instead a notion of coordinate-free conformal net [BDH2], in which the collection of algebras is indexed by abstract intervals, without reference to the circle.  This more flexible notion is reminiscent of structures that have been considered in algebraic quantum field theory on curved 4-dimensional space-times~\cite{[BFV]}.
\begin{definition}
\label{def:conformal-net}
  A coordinate-free conformal net is a continuous 
  covariant functor
\[
\cA \colon \left\{\parbox{3.2cm}{\center \rm oriented intervals, embeddings}\right\} \to \left\{\parbox{4.2cm}{\center \rm von Neumann algebras, injective linear maps}\right\}
\]
taking orientation-preserving embeddings to homomorphisms and orientation-reversing embeddings to antihomomorphisms; here an `interval' is a 1-manifold diffeomorphic to $[0,1]$.
The functor $\cA$ is subject to the following conditions, for $I$ and $J$ subintervals of the interval $K$: 
\begin{list}{$\circ$}{\leftmargin=3ex \rightmargin=2ex \labelsep=1ex \labelwidth=1ex \itemsep=0ex \topsep=1ex}
   \item \emph{Locality:} If $I,J\subset K$ 
     have disjoint interiors, then the images of $\cA(I)$ and $\cA(J)$ are commuting subalgebras of $\cA(K)$.
   \item   \emph{Strong additivity:} If $K = I \cup J$, 
      then the algebra $\cA(K)$ is topologically generated by $\cA(I)$ and $\cA(J)$.
   \item \emph{Split property:} 
      If $I,J\subset K$ are disjoint, 
      then the map from the algebraic tensor product 
      $\cA(I) \otimes_{\mathit{alg}} \cA(J) \to \cA(K)$ extends
      to the spatial tensor product.
   \item \label{def:conformal-net:inner}
      \emph{Inner covariance:}
      If $\varphi \colon I \to I$ is a 
      diffeomorphism that restricts to the 
      identity in a neighborhood of $\partial I$,
      then $\cA(\varphi)$ is an inner automorphism of $\cA(I)$.
   \item \label{def:conformal-net:vaccum} 
      \emph{Vacuum sector:} 
      Suppose $J \subsetneq I$ contains the boundary point
      $p \in \partial I$. Let $\bar{J}$ denote $J$ with the reversed 
      orientation. $\cA(J)$ acts on $L^2(\cA(I))$ 
      via the left action of $\cA(I)$, and 
      $\cA(\bar{J}) \cong \cA(J)^\mathit{op}$ acts on $L^2(\cA(I))$ 
      via the right action of $\cA(I)$.
      We require that the action of
      $\cA(J) \otimes_{\mathit{alg}} \cA( \bar{J} )$ on $L^2(\cA(I))$
      extends to an action of $\cA({J \cup_p \bar{J}})$.
\end{list}
\end{definition}

\noindent In the vacuum sector property in this definition, $L^2(A)$ refers to the Haagerup space, that is the standard form, of the von Neumann algebra $A$~\cite{[Ha]}---this is a Hilbert space, equipped with the structure of an $A$-$A$-bimodule, that functions as an identity bimodule for Connes fusion over $A$.  We do not explicitly demand the existence of a vacuum vector, nor, more dramatically, do we require that the net be positive energy.  Note that classical conformal nets can be promoted to give examples of coordinate-free conformal nets.  Another source of examples is Minkowskian conformal field theories: nets of algebras on 2-dimensional Minkowski space-time provide (non-positive-energy) conformal nets when restricted to a spatial slice~\cite{[KLM],[LR]}.

Morphisms between coordinate-free conformal nets are called defects, and their definition is inspired by the notions of topological and conformal defects in quantum and conformal field theory.
\begin{definition}
Let $\cA$ and $\cB$ be coordinate-free conformal nets.
An $\cA$-$\cB$-defect is a functor
\[
D \colon \left\{\parbox{5.4cm}{\center \rm intervals in $\RR$ whose boundary does not contain $0$, inclusions}\right\} \to \left\{\parbox{4.1cm}{\center \rm von Neumann algebras, homomorphisms}\right\}
\]
such that $D|_{\RR_{> 0}}$ is given by $\cA$, and $D|_{\RR_{< 0}}$ is given by $\cB$.
It is subject to the following conditions:
\begin{list}{$\circ$}{\leftmargin=3ex \rightmargin=1ex \labelsep=1ex \labelwidth=1ex \itemsep=0ex \topsep=1ex}
\item   \emph{Isotony:} If $I\subset J$ and $0\in I$, then the corresponding map $D(I)\to D(J)$ is injective. 
\item \emph{Locality:} If $I\cap J$ is a point, then the images of $D(I)$ and $D(J)$ commute in $D(I\cup J)$.
\item   \emph{Strong additivity:} $D(I\cup J)$ is topologically generated by $D(I)$ and $D(J)$.
\end{list}
The defect functor $D$ is also subject to a version of the vacuum sector axiom, which we do not reproduce here.
\end{definition}

Defects are the analogs for conformal nets of bimodules between algebras.  The basic underlying objects in a geometric model for $K$-theory are bundles, not of bimodules but of modules for Clifford algebras.  When one of the two nets $\cA$ or $\cB$ is trivial (that is, is the constant functor $I \mapsto \CC$), the notion of a defect specializes to the conformal-net analog of a module.  This notion is called a boundary condition for the net, and is given as follows.
\begin{definition}
A boundary condition for the coordinate-free conformal net $\cA$ is a functor
\[
D \colon \left\{\parbox{3cm}{\center \rm intervals in $\RR_{\ge 0}$, inclusions}\right\} \to \left\{\parbox{4.1cm}{\center \rm von Neumann algebras, homomorphisms}\right\}
\]
whose restriction to $\RR_{> 0}$ is given by $\cA$.
It must satisfy isotony, locality, strong additivity, and vacuum sector conditions; the last of these is:
\begin{list}{$\circ$}{\leftmargin=3ex \rightmargin=1ex \labelsep=1ex \labelwidth=1ex \itemsep=0ex \topsep=1ex}
\item \emph{Vacuum sector:}
     Suppose $0\in I\subset \RR_{\ge 0}$. Let $J$ be a subinterval of $I$ that contains the boundary point
      $p \in \partial I$, $p\not = 0$. The algebras $\cA(J)$ and $\cA(\bar{J})$ act on $L^2(D(I))$ 
      via the left and right action of $D(I)$, respectively.
      We require that the action of
      $\cA(J) \otimes_{\mathit{alg}} \cA( \bar{J} )$ on $L^2(D(I))$
      extends to an action of $\cA(J \cup_p \bar{J})$.
\end{list}
\end{definition}

\noindent Note that defects and boundary conditions are well established in the conformal field theory literature~\cite{[FRS],[LR]}.  Previous notions do not, however, permit boundary conditions for chiral conformal field theories, as the above notion does; the basic conformal nets we will use in our investigation of $\TMF$ will necessarily be chiral CFTs.

The notion of boundary condition for a net allows us to consider bundles of boundary conditions as potential geometric representatives of $\TMF$-cohomology classes.  To investigate this idea further, we need a family of nets that functions as an analog of the Clifford algebras.  Indeed we can construct such a family, namely the free fermions.

\subsection*{4. $\KO$-theory and Clifford algebras, $\TMF$ and fermions}

Both in the homotopy-theoretic construction~\cite{[HMi]} and in the field-theoretic approach to cohomology theories~\cite{[ST1]}, there are vivid parallels between real $K$-theory, that is $\KO$, and $\TMF$.  As mentioned above, a crucial ingredient in a geometric understanding of $\KO$ is the family $\Cl(n)$ of Clifford algebras.  We propose that the free fermion coordinate-free conformal nets $\Fer(n)$ play a role for $\TMF$ analogous to the role of Clifford algebras for $\KO$.  Note that the relevant Clifford algebras are naturally $\ZZ/2$-graded algebras, and the fermion nets must similarly be considered as $\ZZ/2$-graded conformal nets---see~\cite{[DH]} for a precise definition of the $\ZZ/2$-graded version of nets.

The family of Clifford algebras has three essential properties: 1) it extends to a functor from real vector spaces to algebras, so in particular the $n^{\text{th}}$ algebra $\Cl(n)$ carries an action of the $n^{\text{th}}$ orthogonal group; 2) this functor is exponential, in that $\Cl(V \oplus W) = \Cl(V) \otimes \Cl(W)$; 3) these algebras encode the fundamental transformation group for $\KO$, in the sense that the automorphism group of a vector space $V$ equipped with a Morita equivalence between $\Cl(V)$ and $\Cl(n)$ is the spin group $\Spin(V)$.  We would like our fermion family of nets to satisfy analogous properties.

As a net, the fermion $\Fer(1)$ provides an algebra $\Fer(1)(I)$ for each interval $I$; this algebra $\Fer(1)(I)$ is defined, roughly, as the completion in a Fock representation of the Clifford algebra of spinor-valued $L^2$ functions on the interval $I$.  The family $\Fer(n)$ of nets does indeed satisfy our three desired properties: 1) by tensoring the spinors with a vector space, the construction produces a functor from vector spaces to nets, $V \mapsto \Fer(V)$; 2) because the Clifford algebra construction on $L^2$ functions is exponential, the fermion construction itself is exponential; 3) the fermions encode the fundamental transformation group for $\TMF$, namely the string group, in the sense that the automorphism group of a vector space $V$ equipped with an invertible defect between $\Fer(V)$ and $\Fer(n)$ is the string group $\String(V)$.  This last property is described in more detail in the next section. \vspace{1ex}

\centerline{
\begin{tabular}{|c|c|}\hline
$\phantom{\Big|}$ Clifford algebra $\cl(n)$ $\phantom{\Big|}$
&
Free fermion $\Fer(n)$\\\hline\hline
$\phantom{\Big|}$ 
$\cl(n)$ has an action of $O(n)$
$\phantom{\Big|}$
& $\Fer(n)$ has an action of $O(n)$\\\hline
$\phantom{\Big|}$ 
$\cl$ is an exponential functor:
$\phantom{\Big|}$
& 
$\Fer$ is an exponential functor: \\
$\cl(V\oplus W)=\cl(V)\otimes\cl(W)$
& 
$\Fer(V\oplus W)=\Fer(V)\otimes\Fer(W)_{\phantom{|_{|_|}}}\!\!\!\!$\\\hline
$\phantom{\Big|}$ 
$\cl(n)$ can be used to define $\Spin(n)$
$\phantom{\Big|}$
&
$\phantom{\Big|}$
$\Fer(n)$ can be used to define $\String(n)$
$\phantom{\Big|}$
\\\hline
\end{tabular}} 
\vspace{1ex} 

Roughly speaking, a class in the $n^{\text{th}}$ real $K$-theory of a space $X$, that is $\KO^n(X)$, is represented by a bundle of $\Cl(n)$-modules over $X$.  A $\Cl(n)$-module can be viewed as a homomorphism, in the 2-category $\Alg$ of algebras, from the Clifford algebra to the trivial algebra; that is, the module is an element of $\Hom_{\Alg}(\Cl(n),1)$.  As the fermions are the $\TMF$ analogs of Clifford algebras, we now consider homomorphisms in the 3-category $\CN$ of nets from the fermions to the trivial net; these elements of $\Hom_{\CN}(\Fer(n),1)$ are boundary conditions for the fermions.  Bundles of such boundary conditions provide, we believe, the core data for a representative of a $\TMF$-cohomology class. \vspace{1ex} 

\centerline{\!\!
\begin{tabular}{|c||c|c|}
\hline
$\phantom{\Big|}$ 
Cohomology theory
$\phantom{\Big|}$ 
&$\KO^*$&$\TMF^*$\\\hline
$\parbox{5cm}{\centerline{The cohomological}
\centerline{degree is controlled by}}^{\phantom{|^{|^I}}}_{\phantom{\big |_|}}$\!\!\!\!\!
&$\parbox{3.7cm}{\centerline{The Clifford algebras}
\centerline{$\cl(n)$}}$
&$\parbox{3cm}{\centerline{The free fermion}
\centerline{conformal nets $\Fer(n)$}}$
\\\hline
$\parbox{5.3cm}{\centerline{Cohomology classes}\vspace{-.05cm} 
\centerline{of degree $n$ are represented by}}^{\phantom{|^{|^|}}}_{\phantom{\big |_|}}$\!\!\!\!&
$\parbox{3cm}{\centerline{Bundles of}\vspace{-.05cm} 
\centerline{$\cl(n)$-modules}}$ &
$\parbox{4.8cm}{\centerline{Bundles of}\vspace{-.05cm} 
\centerline{$\Fer(n)$-boundary conditions}}$
\\\hline
\end{tabular}}
\vspace{1ex} 

It is not, in fact, the case that any class in $\KO^*(X)$ can be represented by a bundle of $\Cl(n)$-modules.  Describing a complete geometric definition of $\KO$ involves elaborating and modifying the idea of bundles of modules, in one of a few possible directions: for instance, one can consider bundles of Hilbert spaces with Clifford algebra actions and compatible fiberwise Fredholm operators, or one can use quasibundles (that is finite-dimensional but not-necessarily locally trivial bundles) of Clifford modules.  A complete geometric description of $\TMF$ classes will certainly involve analogous refinements to the notion of bundles of fermionic boundary conditions.  Though what refinements are needed is at present unresolved, one can use the geometric field-theoretic viewpoint of Stolz and Teichner~\cite{[ST2]} as a source of clear inspiration concerning potential directions for investigation.  Indeed, when we consider the Stolz-Teichner perspective (that $\TMF$ classes should be bundles of twisted two-dimensional geometric field theories) in the context of conformal nets and fermions, the bundle of field theories when evaluated on points would produce a bundle of fermionic boundary conditions.  The kind of information encoded in the 1- and 2-dimensional parts of such field theories provides a guide to potential elaborations of the notion of bundles of boundary conditions.

\subsection*{5. String structures and the $\TMF$-Euler class}

String manifolds have $\TMF$-fundamental classes~\cite{[AHR]}, and in this sense string structures play a role for $\TMF$-cohomology analogous to that played by spin structures for $\KO$ or by orientations for ordinary cohomology.  In the paper~\cite{[DH]} we provide the first direct connection between the conformal net of free fermions and $\TMF$ by proving that the fermions elegantly encode the notion of a string structure on a vector bundle.

Recall that the string group $\String(n)$ is a topological group whose homotopy type is the 3-connected cover of $O(n)$.  It fits into a tower of connective covers of the orthogonal group, obtained by successively killing the lowest remaining homotopy group:
\[
O(n)\underset{\substack{\uparrow\\\text{kill $\pi_0$}}}{\longleftarrow} 
SO(n) \underset{\substack{\uparrow\\\text{kill $\pi_1$}}}{\longleftarrow} 
\Spin(n) \underset{\substack{\uparrow\\\text{kill $\pi_3$}}}{\longleftarrow} 
\String(n).
\]  
There are a number of existing models for the string group, for instance those in~\cite{[ST1],[BCSS],[Wa],[SP]}.  

From a homotopy theoretic perspective, an orientation on a vector bundle is a lift of the classifying map of the bundle from $BO(n)$ to $BSO(n)$; similarly a spin structure or string structure is a lift of the classifying map from $BO(n)$ to respectively $B\Spin(n)$ or $B\String(n)$.  A more concrete geometric notion of orientation is the following: an orientation on a vector bundle is a trivialization of the top exterior power of the bundle.  Not surprisingly, for geometric applications, having this geometric characterization of the notion of orientation is essential.  As we will describe presently, such a characterization exists for spin structures, and one of the contributions of our fermionic approach in~\cite{[DH]} is the construction of an analogously geometric characterization of string structures.

Let $V$ be a real $n$-dimensional vector space equipped with an inner product.  A typical description of a spin structure on $V$ is as a nontrivial double cover of the oriented orthonormal frame bundle of $V$.  This description is reasonably concrete, but not entirely geometric because double covers are themselves a rather topological construction.  We can refine this double-cover notion of spin structure by specifying a geometric structure $S$ associated to the vector space $V$ such that the natural projection map $\pi: \Aut(V,S) \ra \Aut(V)$, from the automorphism group of the pair $(V,S)$ to $\Aut(V) \!=\! SO(V)$, is a nontrivial double cover.  An example of such a structure $S$ is an invertible $\Cl(V)$-$\Cl(n)$ bimodule, that is a Morita equivalence between the Clifford algebras of $V$ and of $\RR^n$; we also insist that this invertible bimodule $S$ is equipped with an inner product and that automorphisms of $S$ be unitary.  To see that the automorphism group of this structure is as desired, observe that
\[
\ker\big(\pi\!:\Aut(V,S) \to \Aut(V)\big)=\Aut_{\cl(V)\text{-}\cl(n)\text{-bimod}}\big(S\hspace{.2mm}\big)=\{\pm 1\}.
\]
A computation of the boundary homomorphism of the sequence $\{\pm 1\} \ra \Aut(V,S) \ra \Aut(V)$ shows that the cover is nontrivial.

The main theorem of~\cite{[DH]} shows that replacing the Clifford algebras by the free fermions in the above description does indeed produce a characterization of a string structure.  Specifically, a string structure on a vector space $V$ is an invertible $\Fer(V)$-$\Fer(n)$ defect $D$.  Here `invertible' must be interpreted as meaning weakly invertible in the 3-category $\CN$ of conformal nets; said differently, the structure is an equivalence in $\CN$ between $V$-fermions $\Fer(V)$ and $n$-fermions $\Fer(n)$.  As in the spin case, to verify that the structure is as desired, one considers the kernel of the projection to the special orthogonal group:
\[
\ker\big(\pi\!:\Aut(V,D) \to \Aut(V)\big) = \Aut_{\Fer(V)\text{-}\Fer(n)\text{-defect}}\big(D\big) 
= \{\text{$\ZZ/2$-gr complex lines}\}.
\]
This kernel $K$ is not a group but a groupoid, and its homotopy type is $\ZZ/2 \times BS^1$, as we would expect for the kernel of the projection from $\String(V)$ to $SO(V)$.  An elaborate computation shows that the boundary homomorphism $\pi_3(SO(n)) \ra \pi_2(K)$ is an isomorphism, for $n$ at least 5, and this implies that the homotopy type of $\Aut(V,D)$ is indeed that of $\String(V)$.

These geometric models for spin and string structures allow a particularly direct description of, respectively, the $\KO$-Euler class of a spin vector bundle and the $\TMF$-Euler class of a string vector bundle.  Let $V$ be a vector bundle over $X$ equipped with a spin structure $S$; that is, $S$ is a bundle over $X$ whose fiber at $x \in X$ is an invertible $\Cl(V_x)$-$\Cl(n)$ bimodule $S_x$.  Recall that certain $\KO^n(X)$ classes can be represented by bundles of $\Cl(n)$-modules.  The $\KO$-Euler class $e_V \in \KO^n(X)$ of the bundle $V$ is represented by the bundle $S$, viewed simply as a right $\Cl(n)$-module bundle:
\[
e_V = \big[S\acted \cl(n)\big]\in \KO^n(X).
\]
That is, to obtain the Euler class of $V$ from the spin structure $S$, we only need to forget the left $\Cl(V)$ action on $S$.

To adapt this description to the $\TMF$-Euler class $e_V \in \TMF^n(X)$ of a string vector bundle $V$ over $X$, we need to know how to take a $\Fer(V)$-$\Fer(n)$-defect and `forget the action of $\Fer(V)$'.  This can be accomplished as follows:
\def\overharp{\overset{\raisebox{-1mm}{$\scriptstyle\rightharpoonup$}}} 
\begin{definition}
Given an $\cA$-$\cB$-defect $E$,
we let $\overharp{E}:\{$\rm intervals in $\RR_{\ge 0}\}\to\{$vN algebras$\}$ \it be the boundary condition for $\cB$ 
given by
\[
\overharp{E}(I):=\begin{cases}
\cB(I)&\text{\rm if\,\, $0\not\in I$,}\\
E\big([-1,0] \cup I\big)&\text{\rm if\,\, $0\in I$.}
\end{cases}
\]
\end{definition}
\noindent Now let $V$ be a vector bundle equipped with a string structure $D$; that is, $D$ is a bundle over $X$ whose fiber at $x \in X$ is an invertible $\Fer(V_x)$-$\Fer(n)$ defect $D_x$.  Forgetting the $\Fer(V_x)$ action on each fiber produces a bundle $\overharp{D}$ of $\Fer(n)$-boundary conditions, with fibers $\overharp{D_x}$.  We expect this bundle $\overharp{D}$ of boundary conditions, perhaps with additional data, will represent the $\TMF$-Euler class $e_V \in \TMF^n(X)$ of the string vector bundle $V$.

\subsection*{6. Periodicity of the fermions}

There are two distinct notions of equivalence between algebras: algebra isomorphism and Morita equivalence.  The first notion arises by viewing algebras as the objects of a 1-category whose morphisms are algebra homomorphisms.  The second notion arises by viewing algebras as the objects of a 2-category whose morphisms are bimodules and whose 2-morphisms are maps of bimodules.

There are, analogously, two distinct notions of equivalence between conformal nets.  The first notion, isomorphism of nets, arises by considering nets as objects of a 1-category, with morphisms the natural transformations of the functors defining the nets.  The existence of a 3-category of nets provides a completely new, second notion of equivalence of nets, which we refer to as $\CN$-equivalence: two nets $\cA$ and $\cB$ are $\CN$-equivalent if there exist defects $D\!: \cA \ra \cB$ and $E\!: \cB \ra \cA$ such that both composites $D E$ and $E D$ are Morita equivalent to the identity defects.

Though the essential nature and properties of $\CN$-equivalence are as yet enigmatic, we can establish certain relationships between the notion of $\CN$-equivalence and the representation categories of nets.  Note that a representation of a net $\cA$ is a Hilbert space equipped with compatible actions of the algebras $\cA(I)$ for $I$ a subinterval of the standard circle.  The representation category $\Rep(\cA)$ of a net $\cA$ can be expressed in purely categorical terms within the 3-category $\CN$ as
\[
\Rep(\cA) = \Hom_{\Hom_\CN(\cA,\cA)}(\id_\cA,\id_\cA).
\]
As a result, the $\CN$-equivalence class of a net encodes the representation category, in the sense that if two nets have inequivalent representation categories, then they cannot be $\CN$-equivalent.  $\CN$-equivalence also provides a specific characterization of the kernel of the functor taking a net to its representation category.  Kawahigashi, Longo, and M\"uger~\cite{[KLM]} showed that a conformal net has trivial representation category if and only if the $\mu$-index of the net is equal to one.  We have proven~\cite{[DH]} that a conformal net $\cA$ has $\mu$-index equal to one if and only if it is $\CN$-invertible, that is if there is another net $\cB$ such that $\cA \otimes \cB$ is $\CN$-equivalent to the trivial net $1$.  The fermion net has trivial representation category, so in particular $\Fer(n)$ is an invertible net.

The real Clifford algebra $\Cl(1)$ is (Morita) invertible, in the 2-category of algebras, and it generates a $\ZZ/8$ subgroup of the group of Morita equivalence classes of invertible algebras; that is, for any $n$, there is a Morita equivalence
\[
\cl(n)\,\,\simeq\,\cl(n+8).
\]
This 8-fold periodicity of the Clifford algebras is the direct algebraic correspondent of the Bott periodicity of $\KO$-theory.  In a curious reversal, the periodicity of $\TMF$ is known~\cite{[HMa],[Ba]}, namely $\TMF^n(X) \simeq \TMF^{n+576}(X)$, but the corresponding algebraic periodicity has not yet been established.  On the view that the fermions are the $\TMF$ analogs of the Clifford algebras, we make the following conjecture.
\begin{conjecture}
For every $n$, there exists a $\CN$-equivalence between the conformal nets
$\Fer(n)$ and $\Fer(n+576)$.
\end{conjecture}

\noindent In their original question asking for a 3-category delooping von Neumann algebras, as described in section 3, Stolz and Teichner indicated that one desiderata for the 3-category would be that it contain a $\ZZ/576$ subgroup in its group of equivalence classes of invertible objects.  If the fermion has minimal period 576, then the fermion conformal nets would provide such a subgroup.

This fermion periodicity conjecture remains quite mysterious, even from the point of view of physics.  As described in the next section, though, we can establish a lower bound of 24 on the periodicity of fermions.

\subsection*{7. Fermions and stable homotopy}

There is an invariant of invertible conformal nets taking values in the $24^{\text{th}}$ roots of unity in $S^1$---this invariant arises from an action of the third homotopy group of the sphere spectrum on the geometric realization of the 3-groupoid of invertible conformal nets.  Our lower bound on the periodicity of the fermions is established by computing that this invariant on $\Fer(1)$ is a primitive $24^\text{th}$ root of unity.

Given a symmetric monoidal 3-category $C$, let $C^\times$ denote the symmetric monoidal 3-groupoid of invertible objects, invertible morphisms, invertible 2-morphisms, and invertible 3-morphisms of $C$.  (If the 3-category $C$ has a natural notion of adjoint on its 3-morphisms, as $\CN$ does, we let $C^\times$ refer to the 3-groupoid whose 3-morphisms are not only invertible but in fact unitary.)  The geometric realization $|C^\times|$ of $C^\times$ will have the structure of a spectrum, and the homotopy groups of this spectrum can be described as
\smallskip\smallskip\\
\indent\smallskip
$\pi_0\big(|C^\times|\big) = \text{equivalence classes of invertible objects of $C$,}$\\\smallskip\indent
$\pi_1\big(|C^\times|\big) = \text{equivalence classes of invertible morphisms from the unit 1 to itself,}$\\\smallskip\indent
$\pi_2\big(|C^\times|\big) = \text{equiv. classes of invertible 2-morphisms from the identity $\id_1$ to itself,}$\\\smallskip\indent
$\pi_3\big(|C^\times|\big) = \text{invertible 3-morphisms from the double identity $\id_{\id_1}$ to itself,}$\\\smallskip\indent
$\pi_n\hspace{-.02cm}\big(|C^\times|\big) = 0\quad \text{for}\quad n\ge 4.$
\smallskip

\noindent As any spectrum is a module over the sphere spectrum $\SS$, the homotopy groups $\pi_*\big(|C^\times|\big)$ are a module over $\pi_*(\SS)$, the ring of stable homotopy groups of spheres.  The most interesting piece of this module structure is the map
\[
\nu:\pi_0\big(|C^\times|\big) \longrightarrow \pi_3\big(|C^\times|\big),
\] 
given by the action of the generator $\nu$ of $\pi_3(\SS) = \ZZ/24$.

As the identity sector on the identity defect on the unit net is simply the Hilbert space of complex numbers, the third homotopy group $\pi_3\big(|\CN^\times|\big)$ of the 3-groupoid of invertible conformal nets is $S^1$.  The above action of the generator $\nu$ of the third stable homotopy group of spheres provides an invariant
\[
\nu:\big\{\text{Invertible conformal nets}\big\}\rightarrow S^1.
\]
Because the class $\nu \in \pi_3(\SS) = \ZZ/24$ has order 24, this invariant necessarily takes values in the $24^\text{th}$ roots of unity.

\begin{theorem}
The image of $\Fer(1)$ under the map $\nu$ is a primitive $24^\text{th}$ root of unity.
\end{theorem}

\begin{proof}
The homotopy groups of $|\CN^\times|$ are as follows:
\bigskip\\
\centerline{\!\!
\begin{tabular}{|c|c|c|c|c|}
\hline
$\pi_0(|\CN^\times|)$&$\pi_1(|\CN^\times|)$&$\pi_2(|\CN^\times|)$&$\pi_3(|\CN^\times|)$&$\pi_{\ge 4}(|\CN^\times|)$\\\hline
$\phantom{\big|}?\phantom{\big|}$&$\ZZ/2$&$\ZZ/2$&$S^1$&$0$\\\hline
\end{tabular}
}\bigskip\\
The first $\ZZ/2$ corresponds to the two Morita equivalence classes of $\ZZ/2$-graded central simple algebras over $\CC$: $\Cl(0)$ and $\Cl(1)$.  The second $\ZZ/2$ corresponds to the two isomorphism classes of $\ZZ/2$-graded lines: the even line and the odd line.  The last non-zero homotopy group $S^1$ corresponds to the linear isometries of $\CC$.

Let $\CN_\mathrm{top}^\times$ be the topological 3-category whose underlying 3-category is $\CN^\times$, and whose spaces of 3-morphisms
have been topologized as subspaces of mapping spaces between Hilbert spaces.
The homotopy groups of the geometric realization of $\CN_\mathrm{top}^\times$ are as follows:
\bigskip\\
\centerline{\!\!
\begin{tabular}{|c|c|c|c|c|c|}
\hline
$\pi_0(|\CN_\mathrm{\!top}^\times|)$&$\pi_1(|\CN_\mathrm{\!top}^\times|)$&$\pi_2(|\CN_\mathrm{\!top}^\times|)$&$\pi_3(|\CN_\mathrm{\!top}^\times|)$&$\pi_{4}(|\CN_\mathrm{\!top}^\times|)$&$\pi_{\ge 5}(|\CN_\mathrm{\!top}^\times|)$\\\hline
$\cong \pi_0(|\CN^\times|)$&$\ZZ/2$&$\ZZ/2$&$0$&$\ZZ$&$0$\\\hline
\end{tabular}
}\bigskip\\
Indeed we have a fibration sequence
\[
|\CN^\times|\to |\CN_\mathrm{\!top}^\times|\to K(\RR,4),
\]
where $\RR$ is given the discrete topology.
Let $F$ denote the image of $\Fer(1)$ in $\pi_0(|\CN^\times|)$.
It follows from the above fiber sequence that the equation $\nu F = e^{2 k \pi i/24}$ in $\pi_3\big(|\CN^\times|\big)$ is equivalent to the Toda bracket relation $k\in \langle 24, \nu, F\rangle$ in $\pi_4(|\CN_\mathrm{\!top}^\times|)$.  (Note that because $\pi_4(\SS) = 0$, this Toda bracket is a $24\ZZ$-torsor inside $\pi_4 (|\CN_\mathrm{\!top}^\times|) \cong \ZZ$.)

We will compute the Toda bracket $\langle 24, \nu, F \rangle$ in $|\CN_\mathrm{\!top}^\times|$ by relating it to a Toda bracket in $\ZZ \times BO$.  The space $\coprod BO(n)$ is a classifying space for real vector spaces, so the free fermion functor $V \mapsto \Fer(V)$, from real vector spaces to conformal nets, induces a map
\[
\coprod_{n\ge 0}BO(n)\to |\CN_\mathrm{\!top}^\times|.
\]
Because the target $|\CN_\mathrm{\!top}^\times|$ is group complete, this map extends to a map of spectra
\[
\rho: \ZZ\times BO\to |\CN_\mathrm{\!top}^\times|
\]
from the group completion $\Omega B(\coprod BO(n)) \simeq \ZZ \times BO$.  This map sends the generator $\iota$ of $\pi_0(\ZZ \times BO)$ to the class of the free fermion $F \in \pi_0(|\CN_\mathrm{\!top}^\times|)$; that is $\rho(\iota) = F$.  The spectrum $\ZZ \times BO$ has the Toda bracket relation $\omega \in \langle 24, \nu, \iota \rangle$ where $\omega \in \pi_4(\ZZ \times BO) \cong \ZZ$ is a generator.  As Toda brackets are natural, we have the relation
\[
\rho(\omega) \in \rho ( \langle 24, \nu, \iota \rangle ) \subset \langle 24, \nu, F \rangle \subset \pi_4 (|\CN_\mathrm{\!top}^\times|).
\]
It therefore suffices to check that $\rho(\omega)$ is a generator of $\pi_4 (|\CN_\mathrm{\!top}^\times|)$.  An involved computation carried out in~\cite{[DH]} shows that $\rho$ induces an isomorphism on $\pi_4$.  Altogether, we conclude that $\nu F = e^{\pm 2\pi i/24}\in \pi_3(|\CN^\times|)\cong S^1$, as desired.
\end{proof}

As the invariant $\nu F \in \pi_3(|\CN^\times|)$ only depends on the $\CN$-equivalence class of the fermion, the above result provides a lower bound on the periodicity of $\Fer(1)$ in the group $\pi_0(|\CN^\times|)$ of invertible conformal nets.
\begin{corollary}
If $\Fer(n)$ is $\CN$-equivalent to the trivial conformal net, then $n$ is a multiple of $24$.
\end{corollary}
In the notes~\cite{[Dr]}, Drinfel{\textquotesingle}d says that there is a braided monoidal category (the Ising category) associated to the free fermion, that is 16-periodic with respect to an operation of `reduced tensor product'.  Combining this mod-16 invariant with the above mod-24 invariant should show that the period of the free fermion, if finite, is in fact a multiple of 48.

\end{document}